\documentclass[letterpaper, 10 pt, conference]{ieeeconf}  % Comment this line out
\IEEEoverridecommandlockouts                              % This command is only
\usepackage{graphics} % for pdf, bitmapped graphics files
\usepackage{epsfig} % for postscript graphics files
\usepackage{mathptmx} % assumes new font selection scheme installed
\usepackage{times} % assumes new font selection scheme installed
\usepackage{amsmath} % assumes amsmath package installed
\usepackage{amssymb}  % assumes amsmath package installed
\usepackage{amsthm}
\usepackage{mathtools}
\usepackage{graphicx}
\usepackage{algorithm}
\usepackage{algorithmic}
\usepackage{float}
\usepackage{hhline}
\usepackage{epstopdf}

\usepackage{multicol}
\title{\LARGE \bf
Successive Convexification of Non-Convex Optimal Control Problems and Its Convergence Properties
}

\author{Yuanqi Mao, Michael Szmuk, and Beh\c{c}et A\c{c}\i kme\c{s}e% <-this % stops a space
\thanks{Yuanqi Mao, Michael Szmuk, and Beh\c{c}et A\c{c}\i kme\c{s}e are with the Department of Aeronautics and Astronautics, University of Washington, Seattle, WA 98105, USA. Emails: {\tt\small yqmao@uw.edu}, {\tt\small mszmuk@uw.edu},  {\tt\small behcet@uw.edu}}%
}

\newtheorem{problem}{Problem}
\newtheorem{lemma}{Lemma}
\newtheorem{theorem}{Theorem}

\newtheorem{definition}{Definition}

\newtheorem*{remark}{Remark}
\newtheorem{assumption}{Assumption}

\DeclareMathOperator*{\esssup}{ess\,sup}

\begin{document}

\maketitle
\thispagestyle{empty}
\pagestyle{empty}

%%%%%%%%%%%%%%%%%%%%%%%%%%%%%%%%%%%%%%%%%%%%%%%%%%%%%%%%%%%%%%%%%%%%%%%%%%%%%%%%
\begin{abstract}

This paper presents  an algorithm to solve non-convex optimal control problems, where non-convexity can arise from nonlinear dynamics,  and non-convex state and control constraints. This paper assumes that  the state and control constraints are already convex or convexified, the proposed algorithm convexifies the nonlinear dynamics, via a linearization,  in a successive manner. Thus at each succession, a convex optimal control subproblem is solved. 
Since the dynamics are linearized and other constraints are convex, after a discretization, the subproblem can be expressed as a finite dimensional convex programming subproblem. Since convex optimization problems can be solved very efficiently, especially with custom solvers, this subproblem can be solved in time-critical applications, such as real-time path planning for autonomous vehicles. Several safe-guarding techniques are incorporated into the algorithm, namely \textit{virtual control} and \textit{trust regions}, which add another layer of algorithmic  robustness. A convergence analysis is presented in continuous-time setting. By doing so, our convergence results will be independent from any numerical schemes used for discretization. Numerical simulations are performed for an illustrative trajectory optimization example.
\end{abstract}

%%%%%%%%%%%%%%%%%%%%%%%%%%%%%%%%%%%%%%%%%%%%%%%%%%%%%%%%%%%%%%%%%%%%%%%%%%%%%%%%
\section{INTRODUCTION} \label{sec:intro}
In this paper, we present a Successive Convexification (\texttt{SCvx}) algorithm to solve non-convex optimal control problems. The problems  involve continuous-time nonlinear dynamics, and possibly non-convex state and control constraints. A large number of real-world problems fall into this category. An example in aerospace applications is the planetary landing problem~\cite{pointing2013,larssys12}. In~\cite{pointing2013}, non-convexity arises from the minimum thrust constraints, while in~\cite{larssys12} nonlinear, time-varying gravity fields and aerodynamic forces consist of additional sources of non-convexity. State constraints can render the problem non-convex as well. One example is the optimal path planning of autonomous vehicles in the presence of obstacles~\cite{richards2002}. Another instance can be found in the highly constrained spacecraft rendezvous and proximity operations~\cite{liu2014solving}.

Such  optimal control problems have been solved  by a variety of approaches~\cite{buskens,gerdtsb}. Most methods  first discretize the problem \cite{hull1997}  and then use  general nonlinear programming solvers to obtain a solution. Unfortunately, general nonlinear optimization have some challenges. First,  there is few known bounds on the computational effort needed. Secondly, they can  become intractable in the sense that a bad initial guess could result in divergence of the numerical algorithm. These two major drawbacks make it unsuitable  for automated solutions and real-time  applications, where computation speed and guaranteed convergence are exactly the two main concerns.
On the other hand, convex programming problems can be reliably solved in polynomial time to global optimality~\cite{BoydConvex}. More importantly, recent advances have shown that these problems can be solved in real-time by both generic Second Order Cone Programming (SOCP) solvers~\cite{alexd}, and by customized solvers which take advantage of specific problem structures~\cite{mattingley2012,dueri2014automated}. This motivates researchers to formulate optimal control problems in a convex programming framework for real-time purposes, e.g. real-time Model Predictive Control (MPC)~\cite{garcia_morari_MPC,MPC_MayneRew00,Zeilinger2014683}.

Although convex programming already suits itself well in solving some optimal control problems with linear dynamics and convex constraints, most real-world problems are not such inherently convex, and thus not readily to be solved. For certain non-convex control constraints, however, recent results  have proven that they can be posed as convex ones without loss of generality via a procedure known as \textit{lossless convexification}~\cite{pointing2013,larssys12,matt_aut1,behcet_aut11}. Certain non-convex state constraints can also be convexified using a successive method as suggested in~\cite{liu2014solving}, provided that they are concave.

This paper aims  to tackle a specific  source of non-convexity: the nonlinear dynamics, which will later be extended (as immediate future work) to non-convexities in state and control constraints that cannot be handled via the methods mentioned earlier. The basic idea is to successively linearize the dynamic equations, and solve a sequence of convex subproblems, specifically SOCPs, as we iterate. While similar ideas in both finite dimensional optimization problems~\cite{griffith1961nonlinear,palacios1982nonlinear,zhang1985improved} and optimal control problems~\cite{rosen1966iterative,machielsen1987,mayne1987exact} have long been tried, few convergence results were reported, and they may suffer from poor performance. Also, different heuristics have been utilized along the way. 

Our \texttt{SCvx} algorithm, on the other hand, presents a systematic way. More importantly, through a continuous-time convergence analysis, we guarantee that the \texttt{SCvx} algorithm will converge, and the solution it converges to will recover optimality for the original problem.
To facilitate convergence, \textit{virtual control} and \textit{trust regions} are incorporated into our algorithm. The former acts like an exact penalty function~\cite{zhang1985improved,mayne1987exact,han1979exact}, but with additional controllability features. The latter is similar to a standard trust-region-type updating rules~\cite{conn2000trust}, but the distinction lies in that we solve each convex subproblem to full optimality to take advantage of powerful SOCP solvers.
%Therefore, our trust regions are updated upon outer iterations, which makes the algorithm more like a generalized trust region approach.
%
To the best of our knowledge, the main contributions of this work are: 1) A novel Successive Convexification (\texttt{SCvx}) algorithm to solve non-convex optimal control problems; 2) A convergence proof in continuous-time.

%%%%%%%%%%%%%%%%%%%%%%%%%%%%%%%%%%%%%%%%%%%%%%%%%%%%%%%%%%%%%%%%%%%%%%%%%%%%%%%%
\section{SUCCESSIVE CONVEXIFICATION}

\subsection{Problem Formulation}
The following nonlinear system dynamics are assumed,
\begin{equation}\label{eq:dynamics}
	\begin{aligned}
	\dot{x}(t)=f(x(t),u(t),t)  \qquad a.e. \quad 0\leq t\leq T,
	\end{aligned}	
\end{equation}
where $ x: \left[ 0,T\right] \rightarrow \mathbb{R}^n $, is the state trajectory, $ u: \left[ 0,T\right] \rightarrow \mathbb{R}^m $, is the control input, and $ f: \mathbb{R}^n \times\mathbb{R}^m \times\mathbb{R} \rightarrow\mathbb{R}^n $, is the control-state mapping, which is Fr\'{e}chet differentiable with respect to all arguments. The state is initialized as $ x_0 $ at time $ t=0 $, i.e. $ x(0)=x_0 $, and the time horizon,  $ T $, is finite. Note that, though $ T $ is  fixed here, free final time problems can be handled under this framework with no additional difficulties. 
The control input  $ u $ is assumed to be Lebesgue integrable on $ \left[ 0,T\right] $, e.g., $ u $ can be measurable and essentially bounded (i.e. bounded almost everywhere) on $ \left[ 0,T\right] $: $ u\in L_\infty[0,T]^m $, with the $ \infty-norm $ defined as $$ \|u\|_\infty \coloneqq \esssup\limits_{t\in[0,T]} \|u(t)\|, $$ 
where $ \|\cdot\| $ is the Euclidean vector norm on $ \mathbb{R}^m $, and $ \esssup $ means the essential supremum.
As a result of the differentiability of function $ f $, $ x $ is continuous on $ \left[ 0,T\right] $, and $ x\in W_{1,\infty}[0,T]^n $, i.e. the space of absolute continuous functions on $ \left[ 0,T\right] $ with measurable and essentially bounded (first order) time derivatives. The \textit{1-norm} of this space is defined by $$ \|x\|_{1,\infty} \coloneqq \max\{\|x\|_\infty,\|\dot{x}\|_\infty\}. $$ It can be shown that, equipped with  these two norms, both $ L_\infty[0,T]^m $ and $ W_{1,\infty}[0,T]^n $ are Banach spaces.

Besides the dynamics, most real-world problems include  control  and state constraints, that must hold for all $ t\in \left[ 0,T\right] $. For simplicity, we assume these constraints are time invariant, i.e.,  $ u(t)\in U $ and $ x(t)\in X $. Here $ U $ and $ X $ are non-convex sets in general, but as we mentioned in \ref{sec:intro}, there are several methods proposed to transform these constraints into convex ones. Further methods of their convexification will be the subject of a future paper. In following, therefore, we assume that they have already been convexified.
Another element is an objective functional (cost), which  is also assumed to be convex. %  that the objective functional is convex, because 
Note that any non-convexity in the cost can be transferred into constraints and convexified afterwards. %In fact, any nonlinearity can also be handled in this way, which essentially leads us to an SOCP formulation.
Hence the following Non-Convex Optimal Control Problem (NCOCP) is considered:
\begin{problem}[\textbf{NCOCP}] \label{prob:NCOCP}
	Determine a control function $ u^*\in L_\infty[0,T]^m $, and a state trajectory $ x^*\in W_{1,\infty}[0,T]^n $, which minimize the functional
	\begin{subequations}
	\begin{equation}
	C(x,u):=	\varphi(x(T),T)+\int_{0}^{T}L(x(t),u(t),t)\;dt, \label{cost:original}
	\end{equation}
	subject to the constraints:
		\begin{align}
			\dot{x}(t)&=f(x(t),u(t),t)  &a.e. \quad 0\leq t\leq T, \label{eq:dynamics2} \\
			u(t)&\in U &a.e. \quad 0\leq t\leq T, \\
			x(t)&\in X &0\leq t\leq T,
		\end{align}
	\end{subequations}
	where $ \varphi: \mathbb{R}^n\times\mathbb{R}\rightarrow\mathbb{R} $, is the terminal cost, 
	$ L: \mathbb{R}^n\times\mathbb{R}^m\times\mathbb{R}\rightarrow\mathbb{R} $, is the running cost, and both are convex and Fr\'{e}chet differentiable;
	$ U\subset\mathbb{R}^m $, and $ X\subset\mathbb{R}^n $. Both are convex and compact sets with nonempty interior.
\end{problem}

\subsection{Algorithm Description}
The only non-convexity remaining in Problem~\ref{prob:NCOCP} lies in the nonlinear dynamics \eqref{eq:dynamics2}. Since it is  an equality constraint, a natural way to convexify it is linearization by using its first order Taylor approximation. The solution to the convexified problem, however, won't necessarily be the same as its non-convex counterpart. To recover optimality, we need to come up with an algorithm that can find a solution which satisfies at least first order optimality condition of the original problem. A natural thought would be doing this linearization successively, i.e. at $ k^{th} $ succession, we linearize the dynamics about the trajectory and the corresponding control computed in the $ (k-1)^{th} $ succession. This procedure is repeated  until convergence. This process essentially forms the basic idea behind our Successive Convexification (\texttt{SCvx}) algorithm.

\subsubsection{Linearization}
Assume the $ (i-1)^{th} $ succession gives us a solution $ (x^{i-1}(t), u^{i-1}(t)) $. Let 
\begin{align*}
A(t)&= \dfrac{\partial}{\partial x}f(x^{i-1}(t),u^{i-1}(t),t), \\
B(t)&= \dfrac{\partial}{\partial u}f(x^{i-1}(t),u^{i-1}(t),t),\\
D(t)&= \dfrac{\partial}{\partial t}f(x^{i-1}(t),u^{i-1}(t),t),
\end{align*}
  $ d(t)=x(t)-x^{i-1}(t) $, and $ w(t)=u(t)-u^{i-1}(t) $, then the first order Taylor expansion about that solution will be
\begin{equation}
	\dot{d}(t)=A(t)d(t)+B(t)w(t)+D(t)+H.O.T.. \label{eq:dyn_lin}
\end{equation}
This is a linear system with respect to $ d(t) $ and $ w(t) $, which are our new states and control respectively. The linearization procedure gets us the benefit of convexity, but it also introduces two new issues, namely \textit{artificial infeasibility} and \textit{approximation error}. We will address them in the following two subsections.

\subsubsection{Virtual Control} \label{sec:virtual}
At various points in the solution space, the above method can generate an infeasible problem, even if the original nonlinear problem itself is feasible. That is  the \textit{artificial infeasibility} introduced by linearization. In such scenarios, the undesirable infeasibility obstructs the iteration process and prevents convergence. %The most evident example of this arises when the problem is linearized about an unrealistically short time horizon, i.e. the final time $ T $ is too small. In such a case, one can intuitively see that there is no feasible control input that can satisfy the prescribed dynamics and constraints.
To prevent this \textit{artificial infeasibility}, we introduce an additional control input $ v(t) $, called \textit{virtual control}, to the linear dynamics \eqref{eq:dyn_lin} (without the  higher order terms):
\begin{equation}
	\dot{d}(t)=A(t)d(t)+B(t)w(t)+E(t)v(t)+D(t), \label{eq:dyn_pen}
\end{equation}
where $ E(t) $ can be chosen based on $ A(t) $ such that the pair $A(\cdot),E(\cdot)$ is controllable. Then, since $ v(t) $ is unconstrained, any state in the feasible region can be reachable in finite time. This is why this \textit{virtual control} can eliminate the \textit{artificial infeasibility}. For example, on  autonomous vehicles, the \textit{virtual control} can be understood as a synthetic acceleration that acts on the vehicle, which can drive the vehicle virtually anywhere in the feasible area.

Since we want to resort to this \textit{virtual control} as needed, it will be heavily penalized via an additional term $ \lambda \gamma(Ev) $ in the cost, where $ \lambda $ is the penalty weight, and $ \gamma(\cdot) $ is the penalty function, defined by
\begin{equation*}
	\gamma(\cdot) \coloneqq \esssup\limits_{t\in[0,T]} \|\cdot(t)\|_1,
\end{equation*}
where $ \|\cdot\|_1 $ is the $ L_1 $ norm on $ \mathbb{R}^n $. For example, $ \|x(t)\|_1 = \sum_{i=1}^{n}|x(t)| $. Thus we have
\begin{equation*}
\gamma(Ev) \coloneqq  \esssup\limits_{t\in[0,T]} \|E(t)v(t)\|_1.
\end{equation*}
Now the penalized cost after linearizing will be defined as
\begin{equation}
L(d,w) \coloneqq 
%\eqref{cost:original}
C(x,u) +\lambda\gamma(Ev), \label{cost:linear}
\end{equation}
while the penalized cost before linearizing can be formulated in a similar fashion:
\begin{equation}
	J(x,u) \coloneqq 
	%\eqref{cost:original}
	C(x,u) +\lambda\gamma(\dot{x}-f). \label{cost:penalty}
\end{equation}
%To be clear, since we are going to solve an SOCP subproblem, here the original cost \eqref{cost:original} is assumed to be already linear.

\subsubsection{Trust Regions}
Another concern when linearizing is potentially rendering the problem unbounded. A simple example will be linearizing the cost $ y_1(x)=0.5x^2 $ at $ x=1 $ to get $ y_2(x)=x-0.5 $. Now if going left is a feasible direction, then the linearized problem could potentially be unbounded while the nonlinear problem will definitely find its minimum at $ x=0 $. The reason behind is: when large deviation is allowed and occurred, the linear approximation sometimes fails to capture the distinction made by nonlinearity, for instance $ y_1(x) $ attains its stationary point at $ x=0 $, while $ y_2(x) $ certainly does not.

%\begin{figure}[!h]
%	\begin{center}
%		\includegraphics[width=0.35\textwidth]{./pics/tr_ex}
%		\caption{Potential unbounded situation and the trust region}
%		\label{fig:tr}
%	\end{center}
%\end{figure}

To mitigate this risk,  we ensure that the linearized trajectory does not deviate significantly from the nominal one obtained in the previous succession, via  a \textit{trust region} on our new control input,
\begin{equation} \label{eq:tr}
	\|w\|_\infty \leq \varDelta,
\end{equation}
and thus our new state will be restricted as well due to the dynamic equations.  The rationale is that we only trust the linear approximation in the thrust region. Continuing with the above example, if we restrict $ x $ to be within the trust region, say $ x\in[0.5, 1.5] $, and adjust the trust region radius $ \varDelta $ as we iterate, it will eventually converge to $ x=0 $.
Fig.~\ref{fig:tr2} shows the typical convergence process of this trust-region type algorithm in solving a 2-D problem. The algorithm can start from virtually anywhere, and manages to converge to a feasible point. Note that the figure also demonstrates \textit{virtual control}, as the trajectory deviates from the constraint in the first few successions.
\begin{figure}[!h]
	\begin{center}
		\includegraphics[width=0.4\textwidth]{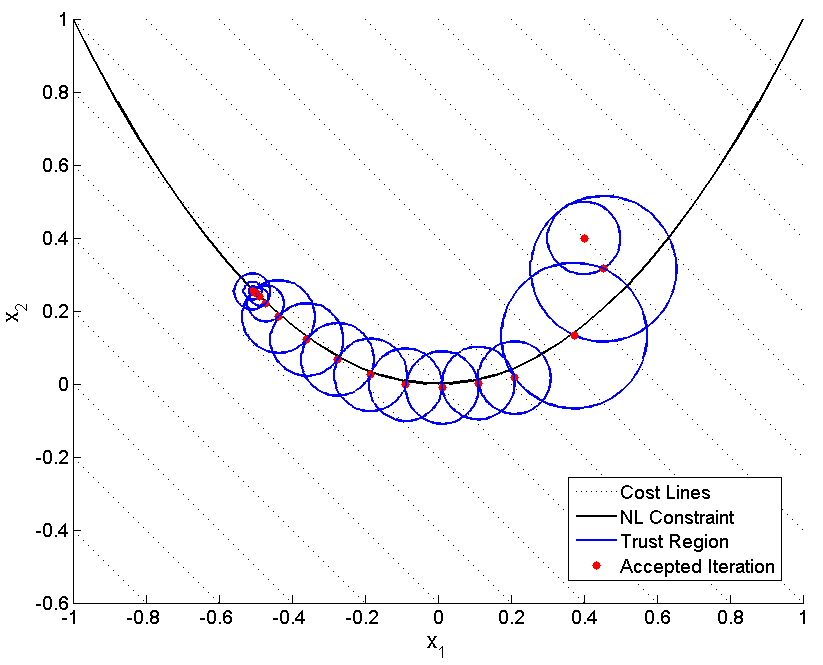}
		\caption{Typical convergence process with \textit{virtual control} and \textit{trust regions}.}
		\label{fig:tr2}
	\end{center}
\end{figure}
%-
\subsubsection{The \texttt{SCvx} Algorithm}
The final problem formulation and the \texttt{SCvx} algorithm can now be presented. Considering the \textit{virtual control} and the \textit{trust regions}, a Convex Optimal Control Problem (COCP) is solved at $ k^{th} $ succession:
\begin{problem}[\textbf{COCP}] \label{prob:COCP}
	Determine $ w^*\in L_\infty[0,T]^m $, and $ d^*\in W_{1,\infty}[0,T]^n $, which minimize the functional $ L(d,w) $ defined in \eqref{cost:linear}, subject to constraints \eqref{eq:dyn_pen}, \eqref{eq:tr}, as well as
	\begin{align*}
	u^{k}(t)+w(t)&\in U,  \\
	x^{k}(t)+d(t)&\in X.
	\end{align*}
\end{problem}
With this convex subproblem,  \texttt{SCvx} algorithm is given  as in Algorithm \ref{algo:SCvx}.
\begin{algorithm}
	\caption{Successive Convexification (\texttt{SCvx})}
	\label{algo:SCvx}
	\begin{algorithmic}
		\STATE \textbf{Input} Select initial state $ x^{1} $ and control $ u^{1} $ s.t. $ x^{1}\in X $ and  $ u^{1}\in U $. Initialize trust region radius with positive $ \varDelta^{1} $ and lower bound $ \varDelta_l $. Select positive penalty weight $ \lambda $, and parameters $ 0<\rho_0<\rho_1<\rho_2<1 $ and $ \alpha>1 $.
		\STATE \textbf{Step 1} At each succession $ k $, solve Problem~\ref{prob:COCP} at $ (x^k,u^k,\varDelta^k) $ to get an optimal solution $ d^k $.
		\STATE \textbf{Step 2} Compute the "actual" change in the penalized cost \eqref{cost:penalty}:
		\begin{equation} \label{eq:actual}
		\Delta J(x^k,u^k) \coloneqq \Delta J^k=J(x^k,u^k)-J(x^k+d^k,u^k+w^k),
		\end{equation}
		and the "predicted" change by linear approximation:
		\begin{equation} \label{eq:predict}
		\Delta L(d^k,w^k) \coloneqq \Delta L^k=J(x^k,u^k)-L(d^k,w^k).
		\end{equation}
		\IF {$ \Delta L^k=0 $} \STATE \textbf{Stop}, and \textbf{return} $ (x^k,u^k) $;
		\ELSE \STATE Compute the ratio $ r^k=\Delta J^k/\Delta L^k $.
		\ENDIF
		\STATE \textbf{Step 3} 
		\IF {$ r^k<\rho_0 $} \STATE Reject this step, contract the trust region radius, i.e. $ \varDelta_k\leftarrow\varDelta_k/\alpha $ and go back to Step 1;
		\ELSE \STATE Accept this step, i.e. $ x^{k+1}\leftarrow x^k+d^k $, $ u^{k+1}\leftarrow u^k+w^k $, and update $ \varDelta^{k+1} $ by
		\begin{equation*}
		\varDelta^{k+1}=\begin{cases}
		\varDelta^k/\alpha, & \text{if }  r^k<\rho_1;\\
		\varDelta^k, & \text{if } \rho_1\leq r^k<\rho_2;\\
		\alpha\varDelta^k, & \text{if } \rho_2\leq r^k.
		\end{cases}
		\end{equation*}
		\ENDIF
		\STATE  $ \varDelta^{k+1}\leftarrow \max\{\varDelta^{k+1}, \varDelta_l \} $, $ k\leftarrow k+1 $, and go to Step 1.
	\end{algorithmic}
\end{algorithm}
This algorithm is of trust region type, and follows standard trust region radius update rules with some  modifications. 

One important distinction lies in the subproblem to be solved at each succession. Conventional trust region algorithms usually perform a line search along the Cauchy arc to achieve a "sufficient" reduction~\cite{conn2000trust}. In our algorithms, however, a full convex optimization problem is solved to speed up the process, i.e. we solve the subproblem to its full optimality. As a result, the number of successions can be significantly less, by achieving more cost reduction at each succession. Admittedly, computational effort may slightly increase at each succession, but thanks to the algorithm  customization techniques \cite{dueri2014automated}, we are able to solve each convex subproblem fast enough to outweigh the negative impact of solving it to full optimality.

In Step 2, the ratio $ r^k $ is used as a metric for the quality of   linear approximations. A desirable scenario is when $ \Delta J^k $  agrees with $ \Delta L^k $, i.e. $ r^k $ is close to 1. Hence if $ r^k $ is above $ \rho_2 $, which means our linear approximation predicts the cost reduction  well, then  we may choose to enlarge the trust region in Step 3, i.e., we  put more faith in our approximation. Otherwise, we may keep the trust region unchanged, or contract its radius if needed. The most unwanted situation is when $ r^k $ is negative, or close to zero. The current step will be rejected in this case, and one has to contract the trust region and re-optimize at $ (x^k,u^k) $.

%%%%%%%%%%%%%%%%%%%%%%%%%%%%%%%%%%%%%%%%%%%%%%%%%%%%%%%%%%%%%%%%%%%%%%%%%%%%%%%%
\section{CONVERGENCE ANALYSIS}
In this section, we present a convergence analysis in the continuous-time setting. Hence we re-formulate the  continuous-time optimal control problem as an infinite dimensional optimization problem in Banach space to simplify notations. To do this, we treat both the states and the control as independent variables, so that in Problem~\ref{prob:NCOCP} and Problem~\ref{prob:COCP}, \eqref{eq:dynamics2} and \eqref{eq:dyn_pen} are considered as equality constraints. For simplicity, {\em we  denote our variable pair $ (x,u) $ as $ x $ in the sequel}. Thus the space X will be re-defined as $$ X \coloneqq W_{1,\infty}[0,T]^n \times L_\infty[0,T]^m. $$ Following this definition, the original dynamics will be represented by a set of algebraic equations:
\begin{equation}
g_{i}(x)=0, \quad i=1,2,\ldots,n, \label{eq:eq_con}
\end{equation}
and the original state and control constraints will be represented by a set of inequalities:
\begin{equation}
	h_{j}(x)\leq 0, \quad j=1,2,\ldots,n+m. \label{eq:ineq_con}
\end{equation}
Denote the set of $ x $ satisfying \eqref{eq:ineq_con} as $ F $,
\begin{equation}
	F=\{x\,|\,h_{j}(x)\leq 0, \quad j=1,2,\ldots,n+m\}. \label{set:feasible}
\end{equation}

For objective functional, the original one \eqref{cost:original} becomes
\begin{equation}
\Psi(x) \coloneqq \varphi(x(T),T)+\int_{0}^{T}L(x(t),t)\;dt. \label{cost:inf_opt}
\end{equation}
Note that $ \Psi(\cdot) $ is Fr\'{e}chet differentiable, since $ \varphi(\cdot) $ and $ L(\cdot) $ are both differentiable. Now Problem~\ref{prob:NCOCP} is ready to be re-formulated as:
\begin{problem} \label{prob:inf_opt}
	Determine $ x\in X $, which minimizes $ \Psi(x) $ defined in \eqref{cost:inf_opt}, subject to constraints \eqref{eq:eq_con} and \eqref{eq:ineq_con}.
\end{problem}

If we rewrite \eqref{eq:eq_con} in a more concise form, $ g(x)=0 $, then the penalized cost defined in \eqref{cost:penalty} will be
\begin{equation}
J(x,\lambda) \coloneqq \Psi(x)+\lambda\gamma(g(x)). \label{cost:inf_pen}
\end{equation}
Then the associated penalty problem can be established in a similar manner:
\begin{problem} \label{prob:inf_pen}
	Determine $ x\in X $, which minimizes $ J(x,\lambda) $ defined in \eqref{cost:inf_pen}, subject to constraints \eqref{eq:ineq_con}.
\end{problem}

\subsection{Exactness of Penalty Function $ \gamma(g(x)) $}
We say a penalty function is \textit{exact}, if there exists a finite $ \lambda $ such that problems before and after penalizing are equivalent in the sense of optimality conditions. A geometric justification of exactness for $ \gamma(g(x)) $ with different norms can be found in~\cite{mayne1987exact}. Alternatively, we can also show it via analysis. To do this, we need to introduce some definitions and assumptions first.
\begin{definition}[\textbf{Local Optimum}] \label{def:local_opt}
	If there exist $ \bar{\lambda}\geq 0 $ and $ \bar{x}\in F $ such that $ J(\bar{x},\lambda)\leq J(x,\lambda) $ for all $ \lambda\geq \bar{\lambda} $ and for all $ x\in N(\bar{x})\cap F $, where $ N(\bar{x}) $ denotes an open neighborhood of $ \bar{x} $, then $ \bar{x} $ is a local optimum of Problem~\ref{prob:inf_opt}.
\end{definition}

\begin{assumption}[\textbf{LICQ}] \label{asup:LICQ}
	Linear Independence Constraint Qualification (LICQ) is satisfied at $ \bar{x} $ if the gradients of active constraints
	\begin{equation*}
		\{\,\nabla g_{1}(\bar{x}),\ldots,\nabla g_{n}(\bar{x}); \nabla h_{i}(\bar{x}), i\in I(\bar{x})\,\}
	\end{equation*}
	are linearly independent, where $ I(\bar{x}) $ is the index set of active inequality constraints,
	\begin{equation*}
	I(\bar{x}) = \{ i\,|\,h_{i}(\bar{x})=0, \quad i=1,2,\ldots,n+m \}.
	\end{equation*}
\end{assumption}

%\begin{definition}[\textbf{Tangent Cone}] \label{def:tan_cone}
%	If LICQ is satisfied, then the tangent cone of active inequality constraints at $ \bar{x} $ is defined by T, where
%	\begin{equation} \label{eq:tan_cone}
%		T=\{ s\,|\,\nabla h_{i}(\bar{x})s\leq 0, \quad i\in I(\bar{x}) \}.
%	\end{equation}
%\end{definition}

Now we are ready to present our first theorem, which states the stationarity part of the first order necessary conditions (KKT conditions) for a point $ \bar{x} $ to be a local solution of Problem~\ref{prob:inf_opt}, as in Definition~\ref{def:local_opt}.
%\begin{theorem} \label{thm:kkt}
%	If the nonlinear constraints of Problem~\ref{prob:inf_opt} satisfy LICQ, and $ \bar{x} $ is a local solution, then there exist Lagrange multipliers $ \mu =(\mu_1,\ldots,\mu_n) $ such that for all $ s\in T $,
%	\begin{equation} \label{eq:kkt}
%		\left( \nabla \Psi(\bar{x})+\sum_{i=1}^{n}\mu_i \nabla g_i(\bar{x})\right) s\geq 0,
%	\end{equation}
%	where $ T $ is the tangent cone as defined in \eqref{eq:tan_cone}.
%\end{theorem}
\begin{theorem}[\textbf{Stationary Conditions}] \label{thm:kkt}
	If the nonlinear constraints of Problem~\ref{prob:inf_opt} satisfy LICQ, and $ \bar{x} $ is a local solution of Problem~\ref{prob:inf_opt}, then there exist Lagrange multipliers $ \mu_i $ and $ \sigma_i\geq 0 $ such that
	\begin{equation*}
	\nabla \Psi(\bar{x}) + \esssup\limits_{t\in[0,T]}\sum_{i=1}^{n}\mu_i \nabla g_i(\bar{x}) + \sum_{i\in I(\bar{x})}\sigma_i \nabla h_i(\bar{x}) = 0.
	\end{equation*}
\end{theorem}

\begin{proof}
	It directly follows Corollary 1 in~\cite{clarke1976new}, with some minor modifications to facilitate the essential supremum, $ \esssup\limits_{t\in[0,T]} (\cdot) $.
\end{proof}

An equivalent form of Theorem~\ref{thm:kkt} is Theorem 3.11 from~\cite{machielsen1987}, which involves a Hamiltonian and is interpreted as a local minimum principle.

To examine the exactness of $ \gamma(g(x)) $, we also need to establish optimality conditions for the penalty problem, i.e. Problem~\ref{prob:inf_pen}. For fixed $ \lambda $, the cost $ J(x) $ is not differentiable everywhere due to non-smoothness of $ \gamma(g(x)) $. However, since $ \Psi(x) $ and $ g(x) $ are both continuously Fr\'{e}chet differentiable, $ J $ is locally Lipschitz continuous. Then we have
\begin{definition}[\textbf{GDD}] \label{def:gdd}
	If $ J $ is locally Lipschitz continuous, the Generalized Directional Derivative (GDD) of $ J $ at $ \bar{x} $ in any direction s exists, and can be defined as
	\begin{equation} \label{eq:gdd}
	dJ(\bar{x},s) \coloneqq \limsup\limits_{\mathclap{\substack{x\rightarrow \bar{x} \\ \delta\rightarrow 0^+}}} \frac{J(x+\delta s)-J(x)}{\delta}.
	\end{equation}
\end{definition}
By~\cite{fletcher1981practical}, $ dJ $ as defined above also satisfies the following implicit relationship,
\begin{equation} \label{eq:gdd2}
	dJ(\bar{x},s)=\max \{\langle\nu, s\rangle\,|\,\nu\in \partial J(\bar{x})\},
\end{equation}
where $ \partial J(\bar{x}) $ is the generalized differential of $ J $ at $ \bar{x} $, i.e.
\begin{equation} \label{eq:gen_diff}
	\partial J(\bar{x})=\{\nu\,|\,dJ(\bar{x},y)\geq \nu y, \forall y\in X \}.
\end{equation}
Applying Theorem 1 in~\cite{clarke1976new} and Using the above definitions, we have
\begin{lemma} \label{lem:stationary}
	If $ \bar{x} $ is a local solution of Problem~\ref{prob:inf_pen}, then there exist multipliers $ \sigma_i\geq 0 $ such that
	\begin{equation} \label{eq:stationary}
		0\in \partial J(\bar{x})+\left\lbrace y\,|\,y=\sum_{i\in I(\bar{x})}\sigma_i \nabla h_i(\bar{x})\right\rbrace \triangleq D(\bar{x}),
	\end{equation}
	where the "$ + $" sign stands for the Minkowski sum of two sets.
\end{lemma}

Define the set of constrained stationary points of Problem~\ref{prob:inf_pen} as $$ S \coloneqq \{x\,|\,x\in F \text{ and } 0\in D(x)\}. $$ Lemma~\ref{lem:stationary} states that if $ \bar{x} $ solves Problem~\ref{prob:inf_pen}, then $ \bar{x}\in S $.

Next lemma gives an explicit expression of $ \partial J(x) $ for any $ x $ and directly follows Theorem 2.1 in~\cite{clarke1975generalized}.
\begin{lemma} \label{lem:differential}
	For any fixed positive weight $ \lambda $, the generalized differential of the penalized cost $ J(x) $ is
	\begin{equation*}
		\partial J(x)=\left\lbrace  \nabla \Psi(x)+\lambda\esssup\limits_{t\in[0,T]}\sum_{i=1}^{n}\mu_i \nabla g_i(x)\,|\,\mu_i \right\rbrace 
	\end{equation*}
	where $ \mu_i $ is defined by
	\begin{equation*}
		\mu_i \begin{cases}
		=\text{sgn }g_i(x), & \text{if } g_i(x)\neq 0; \\
		\in [-1,1], & \text{if } g_i(x)= 0.
		\end{cases}
	\end{equation*}
\end{lemma}

Next theorem gives the main result in this subsection, that is, the exactness of the penalty problem. %The proof of this theorem can be directly modified from Theorem 2.3 in~\cite{zhang1985improved}, and thus ommited here for simplicity.

\begin{theorem} \label{thm:exactness}
	If $ \bar{x} $ is a stationary point of Problem~\ref{prob:inf_opt} with multipliers $ \bar{\mu}_i $, and $ \bar{\sigma}_i $, and if penalty weight $ \lambda $ satisfies
	\begin{equation} \label{eq:exactness}
		\lambda\geq |\bar{\mu}_i|, \quad \forall \; i=1,2,\ldots,n,
	\end{equation}
	then $ \bar{x} $ is a constrained stationary point of Problem~\ref{prob:inf_pen}, that is, $ \bar{x}\in S $.
	
	Conversely, if $ \bar{x}\in S $ is feasible for Problem~\ref{prob:inf_opt}, then it is a stationary point of Problem~\ref{prob:inf_opt}.
\end{theorem}

\begin{proof}
	From Theorem~\ref{thm:kkt}, we have
	\begin{equation} \label{eq:kkt}
		\nabla \Psi(\bar{x}) + \esssup\limits_{t\in[0,T]}\sum_{i=1}^{n}\bar{\mu_i} \nabla g_i(\bar{x}) + \sum_{i\in I(\bar{x})}\bar{\sigma_i} \nabla h_i(\bar{x}) = 0,
	\end{equation}
	and $ \bar{\sigma_i}\geq 0 $. Since $ \bar{x} $ is feasible to Problem~\ref{prob:inf_opt}, we have $ g_i(\bar{x})= 0 $. Therefore from Lemma~\ref{lem:differential}, the generalized differential at $ \bar{x} $ is
	\begin{equation*}
		\partial J(\bar{x})=\left\lbrace  \nabla \Psi(\bar{x})+\lambda\esssup\limits_{t\in[0,T]}\sum_{i=1}^{n}\mu_i \nabla g_i(\bar{x})\,|\,\mu_i \right\rbrace,
	\end{equation*}
	where $ \mu_i \in [-1,1] $. Then from \eqref{eq:exactness}, $ \lambda\geq 1 $. Hence $ \lambda\mu_i $ can take any value in $ \mathbb{R} $, including $ \bar{\mu_i} $. In other words,
	\begin{equation*}
		\nabla \Psi(\bar{x}) + \esssup\limits_{t\in[0,T]}\sum_{i=1}^{n}\bar{\mu_i} \nabla g_i(\bar{x}) \in \partial J(\bar{x}).
	\end{equation*}
	Combining this with $ \bar{\sigma_i}\geq 0 $, from \eqref{eq:kkt} we have $ 0 \in D(\bar{x}) $, which has been defined in \eqref{eq:stationary}. Again, since $ \bar{x} $ is feasible, $ \bar{x} \in F $, i.e. $ \bar{x} \in S $.
	
	Converse result can be trivially derived from Lemma~\ref{lem:stationary} and Lemma~\ref{lem:differential}.
\end{proof}

\begin{remark}
	Although Theorem~\ref{thm:exactness} does not suggest a constructive way to find such a $ \lambda $, it still has important theoretical values. In our current implementations, we select a "sufficiently" large $ \lambda $ and keep it fixed for the whole process. Numerical results show that this works well for most applications.
	
	The "conversely" part of Theorem~\ref{thm:exactness} is what really matters to our subsequent convergence analysis. It guarantees that as long as we can find a stationary point for the penalty problem which is feasible to the original problem, we reach stationarity of the original problem as well.
\end{remark}

\subsection{Convergence Analysis}
Two major convergence results will be stated in this subsection. The first one deals with the case of finite convergence, while the second will handle the situation where an infinite sequence $ \{x^k \} $ is generated by the \texttt{SCvx} algorithm. Note that for simplicity, we suppress the dependence on control variables $ u $ and $ w $ in $ J $ and $ L $ respectively, i.e. $ x $ means $ (x,u) $ and $ d $ stands for $ (d,w) $. For the finite case, we have
\begin{theorem} \label{thm:finite_converge}
	The "predicted" changes $ \Delta L^k $ defined in \eqref{eq:predict} are nonnegative for any $ k $. Furthermore, $ \Delta L^k = 0 $ implies that $ x^{k} $ is a stationary point of Problem~\ref{prob:inf_pen}, i.e. $ x^{k} \in S $.
\end{theorem}

\begin{proof}
	$ d=0 $ is always feasible for the convex subproblem, Problem~\ref{prob:COCP}, while $ d^{k} $ solves this problem. Hence we have
	\begin{equation*}
		L(d^k) \leq L(0) = J(x^k),
	\end{equation*}
	so that $ \Delta L^k=J(x^k)-L(d^k)\geq 0 $, and $ \Delta L^k=0 $ if and only if $ d=0 $ solves problem~\ref{prob:COCP}. In the $ \Delta L^k=0 $ case, one can directly apply Lemma~\ref{lem:stationary} and Lemma~\ref{lem:differential} to get $ x^k \in S $.
\end{proof}

Now the real challenge is when $ \{x^k \} $ is an infinite sequence. It would be nice to have some smoothness when examining the limit process. However in this case, the non-differentiable penalty function $ \gamma(\cdot) $ poses difficulties to our analysis. To facilitate further results, we first note that for any $ x $,
\begin{equation}
	J(x+d) = L(d) + o(\|d\|) \label{eq:taylor}
\end{equation}
and $ o(\|d\|) $ is independent of $ x $. This can be verified by simply writing out the Taylor expansion of $ J(x) $ and using the fact that $ \gamma(\cdot) $ is a linear operator.

Next lemma is the major preliminary result, and its proof also provides some geometric insights of the \texttt{SCvx} algorithm we described. 
% Unlike what we did for Theorem~\ref{thm:finite_converge}, for simplicity we will suppress the dependence on control variables $ u $ and $ w $ in $ J $ and $ L $ respectively, i.e. $ x $ means $ (x,u) $ and $ d $ stands for $ (d,w) $ in future analysis. %

\begin{lemma} \label{lem:main}
	Let $ \bar{x} \in F $ but $ \bar{x} \notin S $, i,e, $ \bar{x} $ is feasible but not a stationary point of the penalty problem, and any scalar $ c\leq 1 $. Then there exist positive $ \bar{\varDelta} $ and $ \bar{\epsilon} $ such that for all $ x \in N(\bar{x},\bar{\epsilon})\cap F $ and $ 0<\varDelta\leq \bar{\varDelta} $, any optimal solution $ \bar{d} $ of the convex subproblem, Problem~\ref{prob:COCP} solved at $ x $ with trust region radius $ \varDelta $ satisfies
	\begin{equation}
		r(x,\varDelta)=\frac{J(x)-J(x+\bar{d})}{J(x)-L(\bar{d})} \geq c,
	\end{equation}
	where $ N(\bar{x},\bar{\epsilon}) $ represents an open neighborhood of $ \bar{x} $ with radius $ \bar{\epsilon} $.
\end{lemma}

\begin{proof}
	Since $ \bar{x} \in F $ but $ \bar{x} \notin S $, we know that $ 0 \notin D(\bar{x}) $, where $ D(\bar{x}) $ is defined in \eqref{eq:stationary}.
	
	By definition in \eqref{eq:gen_diff}, the generalized differential $ \partial J(\bar{x}) $ is the intersection of half spaces, and hence it is a closed convex set. Also, the set
	\begin{equation} \label{eq:conic}
		\left\lbrace y\,|\,y=\sum_{i\in I(\bar{x})}\sigma_i \nabla h_i(\bar{x}), \; \text{for some } \sigma_i \geq 0 \right\rbrace \triangleq C(\bar{x})
	\end{equation}
	is a conic combination of $ \nabla h_i(\bar{x}) $, hence it is a closed convex cone. Since Minkowski sum preserves convexity, $ D(\bar{x}) $ is a closed and convex set as well. Applying the separation theorem of convex sets to $ D(\bar{x}) $, we get that there exists a unit vector $ s $ and a scalar $ \kappa > 0 $ such that for all $ \nu\in D(\bar{x}) $,
	\begin{equation}
		 \langle\nu, s\rangle \leq -\kappa <0. \label{eq:separation}
	\end{equation}
	Since $ \partial J(\bar{x}) $ is a subset of $ D(\bar{x}) $, \eqref{eq:separation} holds for all $ \nu\in\partial J(\bar{x}) $, which means $ \max \{\langle\nu, s\rangle\,|\,\nu\in \partial J(\bar{x})\} \leq -\kappa $. The left hand side is exactly the expression for GDD as described in \eqref{eq:gdd2}. Therefore, we have
	\begin{equation*}
		dJ(\bar{x},s) \coloneqq \limsup\limits_{\mathclap{\substack{x\rightarrow \bar{x} \\ \varDelta\rightarrow 0^+}}} \frac{J(x+\varDelta s)-J(x)}{\varDelta} \leq -\kappa.
	\end{equation*}
	It implies that there exist positive $ \bar{\varDelta} $ and $ \bar{\epsilon} $ such that for all $ x \in N(\bar{x},\bar{\epsilon})\cap F $ and $ 0<\varDelta\leq \bar{\varDelta} $,
	\begin{equation} \label{eq:limit_delta}
		\frac{J(x+\varDelta s)-J(x)}{\varDelta} < -\frac{\kappa}{2}.
	\end{equation}
	Another subset of $ D(\bar{x}) $ is $ C(\bar{x}) $ defined in \eqref{eq:conic}, so \eqref{eq:separation} also holds for all $ \nu \in C(\bar{x}) $. This necessarily implies
	\begin{equation} \label{eq:feasible_s}
		\langle\nabla h_i(\bar{x}), s\rangle \leq 0, \quad i \in I(\bar{x}),
	\end{equation}
	which essentially means $ s $ is a feasible direction of the active constraints at $ \bar{x} $.
	
	Now consider Problem~\ref{prob:COCP} is solved with such $ x $ and $ \varDelta $, and gives an optimal solution $ \bar{d} $. By using \eqref{eq:taylor}, the "actual" change in $ J $ is
	\begin{align}
		\Delta J(x,\bar{d}) &=J(x)-J(x+\bar{d}) \nonumber \\
		&=J(x)-L(\bar{d})-o(\|\bar{d}\|) \nonumber \\
		&=\Delta L(\bar{d})-o(\|\varDelta\|). \label{eq:delta_j}
	\end{align}
	Thus the ratio
	\begin{equation*}
		r(x,\varDelta)=\frac{\Delta J(x,\bar{d})}{\Delta L(\bar{d})} = 1-\frac{o(\|\varDelta\|)}{\Delta L(\bar{d})}.
	\end{equation*}
	 Let $ d' = \varDelta s $. Since $ s $ is a feasible direction as implied by \eqref{eq:feasible_s}, $ d' $ will be a feasible solution to Problem~\ref{prob:COCP}, while $ \bar{d} $ is optimal. Thus we have $ L(\bar{d})\leq L(d') $, which in turn means
	 \begin{equation}
	 	\Delta L(\bar{d}) \geq \Delta L(d') \label{eq:delta_l}
	 \end{equation}
	 
	 Now in the $ \varDelta \rightarrow 0 $ process, $ 0<\varDelta\leq \bar{\varDelta} $ will be satisfied. Then from \eqref{eq:limit_delta}, we have
	 \begin{equation*}
	 	\Delta J(x,d')=J(x)-J(x+d') > \left( \frac{\kappa}{2} \right) \varDelta.
	 \end{equation*}
	 With $ \bar{d} $ replaced by $ d' $ in \eqref{eq:delta_j},
	 \begin{equation*}
	 	\Delta J(x,d')=\Delta L(d')-o(\|\varDelta\|) \geq \left( \frac{\kappa}{2} \right) \varDelta.
	 \end{equation*}
	 Combining this with \eqref{eq:delta_l}, we get
	 \begin{equation*}
	 	\Delta L(x,\bar{d}) \geq \Delta L(d')-o(\|\varDelta\|) > \left( \frac{\kappa}{2} \right) \varDelta.
	 \end{equation*}
	 Thus $ \Delta L(x,\bar{d}) > ( \kappa/2 ) \varDelta + o(\|\varDelta\|) $, and the ratio
	 \begin{equation*}
	 	r(x,\varDelta) = 1-\frac{o(\|\varDelta\|)}{\Delta L(\bar{d})} > 1-\frac{o(\|\varDelta\|)}{( \kappa/2 ) \varDelta + o(\|\varDelta\|)}.
	 \end{equation*}
	 Therefore, as $ \varDelta \rightarrow 0 $, $ r(x,\varDelta) \rightarrow 1 $, and thus for any $ c < 1 $, $ r(x,\varDelta) \geq c $.
\end{proof}

\begin{remark}
	A frustrating situation happens when the algorithm keeps rejecting our trial steps. Fortunately, Lemma~\ref{lem:main} provides some assurance that the \texttt{SCvx} algorithm won't reject the trial steps forever. At some point, usually after reducing $ \varDelta^{k} $ several times, the ratio metric $ r^{k} $ will be greater than any prescribed $ c $, hence $ r^k \geq \rho_{0} $ is guaranteed.
\end{remark}

We can now  present our main result in this section.

\begin{theorem}
	 If the \texttt{SCvx} algorithm generates an infinite sequence $ \{ x^k \} $, then $ \{ x^k \} $ has limit points, and any limit point $ \bar{x} $ is a constrained stationary point of Problem~\ref{prob:inf_pen}, i.e. $ \bar{x} \in S $.
\end{theorem}

\begin{proof}
The proof is by contradiction. Since we have assumed the feasible region to be convex and compact, there is at least one subsequence $ \{ x^{k_{i}} \} \rightarrow \bar{x} $, which is NOT a stationary point. From Lemma~\ref{lem:main},  there exist positive $ \bar{\varDelta} $ and $ \bar{\epsilon} $ such that
	\begin{equation*}
	r(x,\varDelta) \geq \rho_0 \quad \text{for all } x \in N(\bar{x},\bar{\epsilon})\cap F \text{ and } 0<\varDelta\leq \bar{\varDelta}.
	\end{equation*}
	Without loss of generality, we can suppose the whole subsequence $ \{ x^{k_{i}} \} $ is in $ N(\bar{x},\bar{\epsilon}) $, so
	\begin{equation} \label{eq:accept}
	r(x^{k_{i}},\varDelta) \geq \rho_0 \quad \text{for all } 0<\varDelta\leq \bar{\varDelta}.
	\end{equation}
	
	If the initial trust region radius is less than $ \bar{\varDelta} $, then \eqref{eq:accept} will be trivially satisfied.
	
	Now if the initial radius is greater than $ \bar{\varDelta} $, then $ \varDelta $ may need to be reduced several times  before the condition \eqref{eq:accept} is met. Let $ \hat{\varDelta} $ be the last $ \varDelta^{k_{i}} $ that needs to be reduced, then evidently $ \hat{\varDelta} > \bar{\varDelta} $. Also, let $ \varDelta^{k_{i}} $ be the subsequence consists of all the valid radius after the last rejection. Then
	\begin{equation*}
		\varDelta^{k_{i}} = \hat{\varDelta}/\alpha >  \bar{\varDelta}/\alpha.
	\end{equation*}
	Since there is also a lower bound $ \varDelta_{l} $ on $ \varDelta $, we have
	\begin{equation} \label{min_delta}
		\varDelta^{k_{i}} \geq \min \{\varDelta_{l}, \bar{\varDelta}/\alpha\} \triangleq \delta.
	\end{equation}
	Note that condition \eqref{eq:accept} can be express as
	\begin{equation} \label{eq:r_alter}
		J(x^{k_{i}})-J(x^{k_{i}+1}) \geq \rho_{0} \Delta L^{k_{i}}.
	\end{equation}
	
	Our next goal is to find a lower bound for $ \Delta L^{k_{i}} $. Further, we let $ \bar{d} $ solves the convex subproblem, Problem~\ref{prob:COCP} at $ \bar{x} $ with trust region radius $ \delta/2 $, i.e. $ \|\bar{d}\|_\infty \leq \delta/2 $. Let $ \hat{x}=\bar{x}+\bar{d} $, then
	\begin{equation} \label{eq:norm1}
		\|\hat{x}-\bar{x}\|_\infty \leq \delta/2.
	\end{equation}
	Since $ \bar{x}\notin S $, Theorem~\ref{thm:finite_converge} implies
	\begin{equation*}
		\Delta L(\bar{d})=J(\bar{x}) - L(\bar{d}) \triangleq \theta > 0.
	\end{equation*}
	By continuity of $ J $ and $ L $, there exists an $ i_0 > 0 $ such that for all $ i \geq i_0 $
	\begin{equation} \label{eq:r_continuity}
		J(x^{k_{i}}) - L(\hat{x}-x^{k_{i}}) > \theta/2, \text{ and}
	\end{equation}
	\begin{equation} \label{eq:norm2}
		\|x^{k_{i}} - \bar{x}\|_\infty < \delta/2.
	\end{equation}
	From \eqref{eq:norm1} and \eqref{eq:norm2}, we have for all $ i \geq i_0 $
	\begin{equation} \label{eq:tr_ki}
		\|\hat{x}-x^{k_{i}}\|_\infty \leq \|\hat{x}-\bar{x}\|_\infty + \|x^{k_{i}} - \bar{x}\|_\infty < \delta \leq \varDelta^{k_{i}}.
	\end{equation}
	where the last inequality comes from \eqref{min_delta}. Let $ \hat{d}^{k_{i}}=\hat{x}-x^{k_{i}} $, \eqref{eq:tr_ki} implies $ \hat{d}^{k_{i}} $ is a feasible solution for the convex subproblem at $ (x^{k_{i}},\varDelta^{k_{i}}) $ when $ i \geq i_0 $. Then if $ d^{k_{i}} $ is the optimal solution to this subproblem, we have $ L(d^{k_{i}}) \leq L(\hat{d}^{k_{i}}) $, so
%	\begin{align}
%		\Delta L(d^{k_{i}}) &= J(x^{k_{i}}) - L(d^{k_{i}}) \nonumber \\
%		&\geq J(x^{k_{i}}) - L(\hat{d}^{k_{i}}) \nonumber \\
%		&> \theta/2. \label{eq:theta/2}
%	\end{align}
	\begin{equation} \label{eq:theta/2}
		\Delta L(d^{k_{i}}) = J(x^{k_{i}}) - L(d^{k_{i}}) \geq J(x^{k_{i}}) - L(\hat{d}^{k_{i}}) > \theta/2.
	\end{equation}
	The last inequality is due to \eqref{eq:r_continuity}. Combine \eqref{eq:r_alter} and \eqref{eq:theta/2}, we have for all $ i \geq i_0 $
	\begin{equation} \label{eq:last_theta/2}
		J(x^{k_{i}})-J(x^{k_{i}+1}) \geq \rho_{0} \theta/2.
	\end{equation}
	However, the positive infinite series
	\begin{align*}
		\sum_{i=1}^{\infty} \left( J(x^{k_{i}})-J(x^{k_{i}+1}) \right) &\leq 
		\sum_{i=1}^{\infty} \left( J(x^{k_{i}})-J(x^{k_{i+1}}) \right) \\
		&=J(x^{k_{1}})-J(\bar(x)) \leq \infty.
	\end{align*}
	Thus it is convergent, so necessarily $$ J(x^{k_{i}})-J(x^{k_{i}+1}) \;\rightarrow\; 0, $$
	which contradicts \eqref{eq:last_theta/2}. Therefore, we conclude that every limit point $ \bar{x} \in S$.
\end{proof}

%%%%%%%%%%%%%%%%%%%%%%%%%%%%%%%%%%%%%%%%%%%%%%%%%%%%%%%%%%%%%%%%%%%%%%%%%%%%%%%%
\section{NUMERICAL SIMULATIONS}

This section presents a numerical example to demonstrate the proposed algorithm's capabilities. The example consists of a mass subject to ``double-integrator" dynamics with nonlinear aerodynamic drag. The mass is controlled using a thrust input, $\mathbf{T}$, that is limited to a maximum magnitude of $T_{max}$. The problem formulation is summarized below, and its parameters  are given in Table~\ref{t:sim_params}.
\begin{align*}
\text{minimize} & \int_0^{t_f}{\Gamma(\tau)d\tau} \quad \text{s.t.} \\
\mathbf{x}(0) &= \mathbf{x}_i \quad\quad\,\ \mathbf{v}(0) = \mathbf{v}_i \\
\mathbf{x}(t_f) &= \mathbf{x}_f \quad\quad \mathbf{v}(t_f) = \mathbf{v}_f \\
\dot{\mathbf{x}}(t) &= \mathbf{v}(t) \quad\;\;\; \dot{\mathbf{v}}(t) = \mathbf{a}(t) \\
\mathbf{a}(t) &= \frac{1}{m}\left(\mathbf{T}(t)-k_d||\mathbf{v}(t)||\mathbf{v}(t)\right) \\
||\mathbf{T}(t)|| &\le \Gamma(t) \le T_{max}.
\end{align*}

\begin{table}[h!]
	\begin{center}
		\caption{Simulation Parameters}
		\begin{tabular}{lll}
			\hhline{===}
			Parameter         & Value \\
			\hline
			$t_f$             & 10 \\
			$m$               & 1.0 \\
			$k_d$             & 0.25 \\
			$T_{max}$         & 2.0 \\
			$\mathbf{x}_i$    & $[0\quad0]^T$ \\
			$\mathbf{x}_f$    & $[10\quad10]^T$ \\
			$\mathbf{v}_i$    & $[5\quad0]^T$ \\
			$\mathbf{v}_f$    & $[5\quad0]^T$ \\
			$\Delta_l$        & $0.0$ \\
			$\rho_0$          & $0.0$ \\
			$\rho_1$          & $0.25$ \\
			$\rho_2$          & $0.9$ \\
			$\alpha$          & $2.0$ \\
			\hhline{===}
			\label{t:sim_params}
		\end{tabular}
	\end{center}
	\vspace{-20pt}
\end{table}
Here, we linearize the first succession about a constant velocity straight-line trajectory from the initial  to the final position. The convergence history is presented in Fig.~\ref{f:sim_position}, where the first 10 accepted successions are shown, after which a desirable convergence is achieved. %Subsequent ones were omitted as they were too close to be discerned. The figure shows the solution indeed converges to a feasible trajectory.
\begin{figure}[h!]
	\begin{center}
		\includegraphics[width=0.45\textwidth]{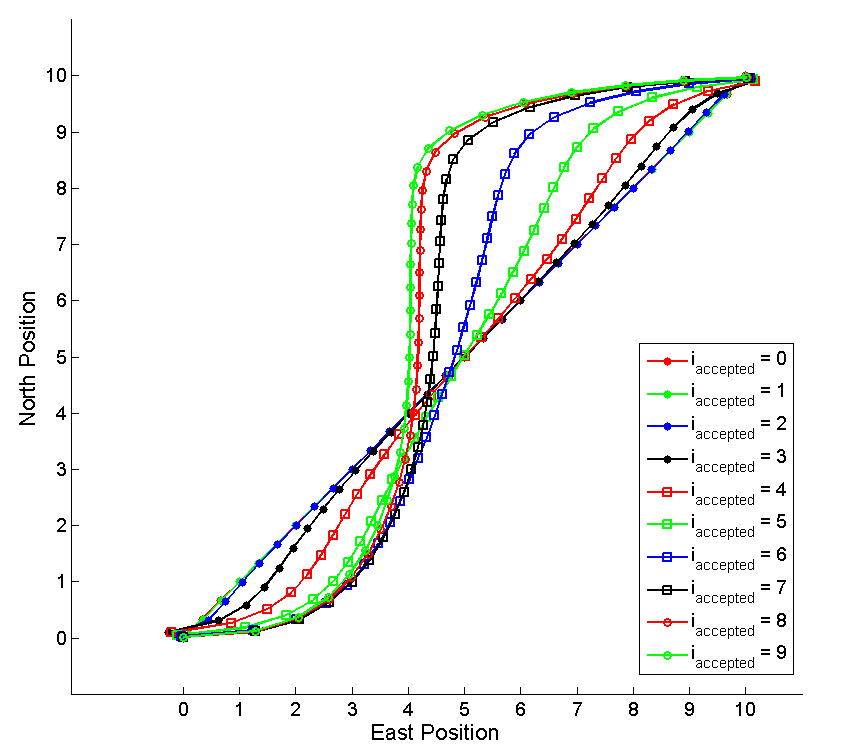}
		\caption{Convergence history of two-dimensional trajectory. \fontseries{sb}\it The dots along the trajectory indicate time-discretization points.}
		\label{f:sim_position}
	\end{center}
	\vspace{-14pt}
\end{figure}
 Fig.~\ref{f:sim_conv_thrust} shows the converged thrust magnitude profile. The plot also shows a no-drag thrust magnitude profile for the same scenario. Note that the no-drag case renders the problem convex, thus recovering the classical ``bang-bang'' solution. In contrast, the case with aerodynamic drag takes advantage of drag at the beginning by decreasing its thrust magnitude, and must sustain a non-zero thrust in the middle of the trajectory to counter the energy loss induced by drag. Note that the area under the thrust curve in the case with drag is less than that of the case without drag. Thus, we see evidence of a situation where the optimal cost of the problem is reduced due to the inclusion of nonlinear dynamics.

\begin{figure}[h!]
	\begin{center}
		\includegraphics[width=0.45\textwidth]{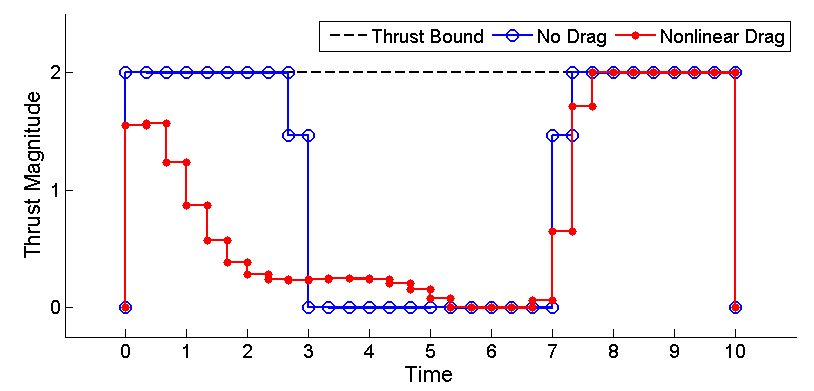}
		\caption{Convergence history of solution and thrust profile: \fontseries{sb}\it The thrust magnitude profile of the converged solution (red) and of the corresponding no-drag problem (blue). %This plot shows the zero-order-hold assumed for the control input between time steps. 
		The markers indicate time-discretization points.}
		\label{f:sim_conv_thrust}
	\end{center}
	\vspace{-14pt}
\end{figure}

Fig.~\ref{f:sim_algorithm} show the process the algorithm undergoes in obtaining the solution. The top plot shows how $r^k$ varies as a function of iteration number, and how the trial step is rejected when $r^k$ drops below $\rho_0$. Note that the trust region radius increases whenever $r^k$ is greater than $\rho_2$, stays the same when $\rho_1 \le r^k \le \rho_2$ or $r^k \le \rho_0$, and decreases when $\rho_0 \le r^k \le \rho_1$. More  importantly, we can see how the process of rejecting the step results in decreasing the trust region radius, $\Delta^k$, thus increasing $r^k$ back above $\rho_0$.

\begin{figure}[h!]
	\begin{center}
		\includegraphics[width=0.45\textwidth]{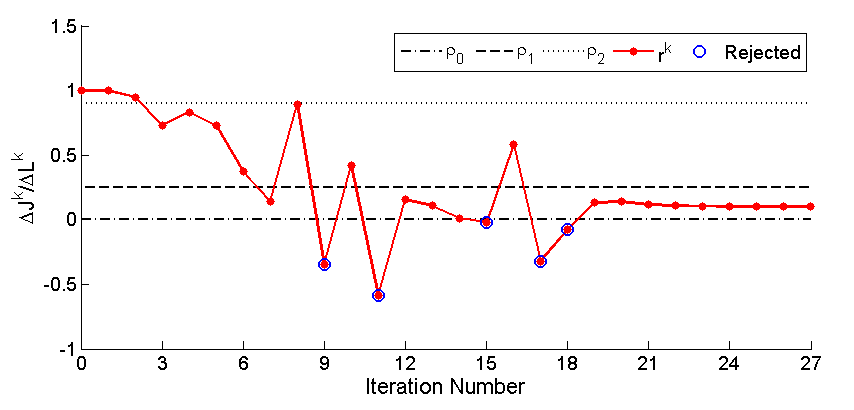}
		\caption{Algorithm behavior during the solution process. \fontseries{sb}\it The plot shows the time history of $r^k$ relative to the thresholds $\rho_0$, $\rho_1$, and $\rho_2$, and indicates the rejected steps.}
		\label{f:sim_algorithm}
	\end{center}
\vspace{-14pt}
\end{figure}

\addtolength{\textheight}{-3cm}

%%%%%%%%%%%%%%%%%%%%%%%%%%%%%%%%%%%%%%%%%%%%%%%%%%%%%%%%%%%%%%%%%%%%%%%%%%%%%%%%
\section{CONCLUSIONS}

In this paper, we have proposed the \texttt{SCvx} algorithms. It successively convexifies nonlinear dynamics, and enables us to solve non-convex optimal control problems in real-time by actually solving a sequence of convex subproblem. We have given a relatively thorough description and analysis of the proposed algorithm, with the aid of a simple, yet illustrative numerical example. 
Convergence has been demonstrated in theoretical proofs, as well as illustrative figures and plots. The main result is: for the infinite convergent case, we have proved that every limit point will satisfied the optimality condition for the original non-convex optimal control problems.
A key next step is to incorporate this algorithm with other convexification techniques to construct a comprehensive algorithmic framework to deal with a  large class of non-convexities.
%We are currently developing a new method to convexify state constraints, which is designed to be a perfect fit for this successive paradigm.

%%%%%%%%%%%%%%%%%%%%%%%%%%%%%%%%%%%%%%%%%%%%%%%%%%%%%%%%%%%%%%%%%%%%%%%%%%%%%%%%
\subsection*{Acknowledgments}
The authors gratefully acknowledge John Hauser of University of Colorado for his deep and valuable insights.

%%%%%%%%%%%%%%%%%%%%%%%%%%%%%%%%%%%%%%%%%%%%%%%%%%%%%%%%%%%%%%%%%%%%%%%%%%%%%%%%

\bibliographystyle{IEEEtran}
\bibliography{SCvx}

\end{document}